\documentclass[12pt,twoside]{amsart}
\usepackage{amssymb,amscd,amsxtra,calc}
\usepackage{amsmath}
\usepackage{xypic}
\usepackage[graph]{xy}

\title{Log del Pezzo surfaces with large volumes}
\author{Kento Fujita} 
\date{\today}
\subjclass[2010]{Primary 14J26; Secondary 14E30}
\keywords{del Pezzo surface, rational surface, extremal ray.}
\address{Research Institute for Mathematical Sciences, 
Kyoto University, Kyoto 606-8502 Japan}
\email{fujita@kurims.kyoto-u.ac.jp}

\newcommand{\pr}{\mathbb{P}}

\newcommand{\Z}{\mathbb{Z}}
\newcommand{\Q}{\mathbb{Q}}
\newcommand{\R}{\mathbb{R}}
\newcommand{\C}{\mathbb{C}}
\newcommand{\F}{\mathbb{F}}

\newcommand{\Pic}{\operatorname{Pic}}

\newcommand{\discrep}{\operatorname{discrep}}

\newcommand{\mult}{\operatorname{mult}}
\newcommand{\coeff}{\operatorname{coeff}}

\newcommand{\sI}{\mathcal{I}}
\newcommand{\sC}{\mathcal{C}}
\newcommand{\sO}{\mathcal{O}}

\newcommand{\sF}{\mathcal{F}}

\newcommand{\dm}{\mathfrak{m}}

\newtheorem{thm}{Theorem}[section]
\newtheorem{lemma}[thm]{Lemma}
\newtheorem{proposition}[thm]{Proposition}
\newtheorem{corollary}[thm]{Corollary}
\newtheorem{claim}[thm]{Claim}

\theoremstyle{definition}
\newtheorem{definition}[thm]{Definition}
\newtheorem{remark}[thm]{Remark}

\newtheorem{example}[thm]{Example}
\newtheorem*{ack}{Acknowledgments} 
\newtheorem*{nott}{Notation and terminology}

\begin{document}

\maketitle 

\begin{abstract}
We classify all of the log del Pezzo surfaces $S$ of index $a$ 
such that the volume $(-K_S^2)$ is larger than or equal to $2a$. 
\end{abstract}


\section{Introduction}\label{intro_section}

The purpose of this paper is to classify all of the log del Pezzo surfaces of fixed 
index with large anti-canonical volumes. 
A normal projective variety $S$ is called a \emph{log Fano variety} if 
$S$ is log-terminal and the anti-canonical divisor $-K_S$ is ample. 
We call $2$-dimensional log Fano varieties as \emph{log del Pezzo surfaces}. 
For a log Fano 
variety $S$, the \emph{index} of $S$ is defined as the smallest positive integer $a$ 
such that $-aK_S$ is Cartier. Log del Pezzo surfaces $S$ with fixed index $a$ 
have been treated by many authors in various viewpoint. 

We recall the viewpoint in terms of the classification problem of log del Pezzo surfaces. 
If $a=1$ then all such $S$ are classified by several authors, 
see \cite{brenton, demazure, HW}; 
if $a=2$ then all such $S$ are classified in \cite{AN1, AN2, AN3, N}; and 
if $a=3$ then all such $S$ are classified in \cite{FY}. 
However, it is not realistic to classify \emph{all} of the log del Pezzo surfaces of 
(fixed) index $a\geq 4$ since even for the case $a=3$ there are $300$ types of such 
log del Pezzo surfaces (see \cite{FY}). 

On the other hand, the problem of the boundedness for log Fano varieties of fixed 
index $a$ is also important in the area of minimal model program. 
The problem is equivalent to bound the anti-canonical volume $(-K_S^n)$ for such $S$ 
if the characteristic of the base field is zero. 
Such problem is called the Batyrev conjecture. See \cite{borisov, HMX}. 
We remark that the authors in \cite{HMX} showed that the Baryrev conjecture is true 
(in any dimension, if the characteristic of the base field is zero). 
Nowadays, a generalized version of the Batyrev conjecture, called the BAB conjecture, 
is considered by many authors, see \cite{BB, alexeev, AM, lai, jiang}. 
The BAB conjecture is true for log del Pezzo surfaces 
but is open in higher-dimensional case. Note that, as a corollary of \cite{jiang}, 
any log del Pezzo surface $S$ of index $a\geq 2$ satisfies that 
$(-K_S^2)\leq(2a^2+4a+2)/a$ and equality holds if and only if $S$ is isomorphic to 
the weighted projective plane $\pr(1,1,2a)$. 

In this paper, we classify the log del Pezzo surfaces $S$ of index $a\geq 4$ such that 
$(-K_S^2)\geq 2a$. The motivation is 
related to both the classification problem and the boundedness 
problem. We note that log del Pezzo surfaces $S$ of index $a\leq 3$ are completely 
classified.

\begin{thm}\label{mainthm}
Let $S$ be a log del Pezzo surface of index $a\geq 4$. 
Then $(-K_S^2)\geq 2a$ holds if and only if $S$ is isomorphic to the 
log del Pezzo surface associated to the $a$-fundamental multiplet of length 
$b$ with $b=\lfloor(a+1)/2\rfloor$ whose type is one of the following: 
\begin{enumerate}
\renewcommand{\theenumi}{\arabic{enumi}}
\renewcommand{\labelenumi}{(\theenumi)}
\item\label{mainthm1}
$\langle${\rm O}$\rangle_a$\,\,\,
$\left((-K_S^2)=(2a^2+4a+2)/a\right)$,
\item\label{mainthm2}
$\langle${\rm I}$\rangle_a$\,\,\,
$\left((-K_S^2)=(2a^2+3a+2)/a\right)$,
\item\label{mainthm3}
$\langle${\rm II}$\rangle_a$\,\,\,
$\left((-K_S^2)=(2a^2+2a+2)/a\right)$,
\item\label{mainthm4}
$\langle${\rm III}$\rangle_a$\,\,\,
$\left((-K_S^2)=(2a^2+a+2)/a\right)$,
\item\label{mainthm5}
$\langle${\rm A}$\rangle_5$ \, $($if $a=5)$ \,\,
$\left((-K_S^2)=54/5\right)$,
\item\label{mainthm6}
$\langle${\rm IV}$\rangle_a$\,\,\,
$\left((-K_S^2)=(2a^2+2)/a\right)$,
\item\label{mainthm7}
$\langle${\rm B}$\rangle_4$ \, $($if $a=4)$ \,\,
$\left((-K_S^2)=8\right)$,
\item\label{mainthm8}
$\langle${\rm C}$\rangle_4$ \, $($if $a=4)$ \,\,
$\left((-K_S^2)=8\right)$.
\end{enumerate}
We describe the above types in Section \ref{ex_section1}. 
\end{thm}

By Theorem \ref{mainthm} and the arguments in 
Section \ref{ex_section}, we have the following result. We note that 
we do not use the results in \cite{jiang} in order to prove Corollary \ref{maincor}. 

\begin{corollary}\label{maincor}
Let $S$ be a log del Pezzo surface of index $a\geq 2$. 
Then $(-K_S^2)>2(a+1)$ holds if and only if $S$ is isomorphic to one of the following 
$($We describe the following varieties in Section \ref{ex_section2}.$):$
\begin{enumerate}
\renewcommand{\theenumi}{\arabic{enumi}}
\renewcommand{\labelenumi}{(\theenumi)}
\item\label{maincor1}
$\pr(1,1,2a)$ \,\,\, $\left((-K_S^2)=(2a^2+4a+2)/a\right)$, 
\item\label{maincor2}
$S_{{\rm I}, a}$ \,\,\, $\left((-K_S^2)=(2a^2+3a+2)/a\right)$, 
\item\label{maincor3}
$S_{{\rm II}_1, a}$ \,\,\, $\left((-K_S^2)=(2a^2+2a+2)/a\right)$, 
\item\label{maincor4}
$S_{{\rm II}_2, a}$ \,\,\, $\left((-K_S^2)=(2a^2+2a+2)/a\right)$, 
\item\label{maincor5}
$\pr(1,1,3)$ \, $($if $a=3)$ \,\, $\left((-K_S^2)=25/3\right)$. 
\end{enumerate}
We remark that all of them are toric varieties. 
\end{corollary}

The strategy for the classification is essentially same as the strategy in \cite{FY} 
based on the earlier work in \cite{N}. First, we reduce the 
classification problem of log del Pezzo surfaces of index $a$ 
to the classification problem of the pair of its minimal resolution and $(-a)$ times 
the discriminant divisor. We call such pair an \emph{$a$-basic pair} 
(see Section \ref{dPb_section}). Next, by contracting $(-1)$-curves very carefully, 
we get an \emph{$a$-fundamental multiplet} from an $a$-basic pair 
(see Section \ref{mult_section}). See also the flowcharts in \cite[\S1]{FY}. 
From the assumption $(-K_S^2)\geq 2a$, the structure of the associated 
$a$-fundamental multiplet has very special structure. 
Thus we can get the multiplets in Section \ref{ex_section1}.

\begin{ack}
The author is partially supported by a JSPS Fellowship for Young Scientists. 
\end{ack}

\begin{nott}
We work in the category of algebraic (separated and of finite type) scheme over a 
fixed algebraically closed field $\Bbbk$ of arbitrary characteristic. 
A \emph{variety} means a reduced and irreducible algebraic scheme. A \emph{surface} 
means a two-dimensional variety. 

For a normal variety $X$, we say that $D$ is a 
\emph{$\Q$-divisor} (resp.\ \emph{divisor} or \emph{$\Z$-divisor}) 
if $D$ is a finite sum $D=\sum a_iD_i$ where $D_i$ are prime divisors and 
$a_i\in\Q$ (resp.\ $a_i\in\Z$). 
For a $\Q$-divisor $D=\sum a_iD_i$, the value $a_i$ is denoted by $\coeff_{D_i}D$ and 
set $\coeff D:=\{a_i\}_i$. For an effective $\Q$-divisor or a scheme $D$ on $X$, 
let $|D|$ be the support of $D$. 
A normal variety $X$ is called \emph{log-terminal} if the canonical divisor $K_X$ is 
$\Q$-Cartier and the discrepancy $\discrep(X)$ of $X$ is bigger than $-1$ 
(see \cite[\S 2.3]{KoMo}).
For a proper birational morphism $f\colon Y\to X$ between normal varieties 
such that both $K_X$ and $K_Y$ are $\Q$-Cartier, we set 
\[
K_{Y/X}:=\sum_{E_0\subset Y \,\,f\text{-exceptional}}a(E_0, X)E_0, 
\]
where $a(E_0, X)$ is the discrepancy of $E_0$ with respects to $X$ 
(see \cite[\S 2.3]{KoMo}). (We note that 
if $aK_X$ and $aK_Y$ are Cartier for $a\in\Z_{>0}$, then 
$aK_{Y/X}$ is a $\Z$-divisor.)

For a nonsingular surface $S$ and a projective curve $C$ which is a closed subvariety 
of $S$, the curve $C$ is called a \emph{$(-n)$-curve} if $C$ is a nonsingular rational 
curve and $(C^2)=-n$. For a birational map $M\dashrightarrow S$ 
between normal surfaces and for a curve $C\subset S$, 
the strict transform of $C$ on $M$ is denoted by $C^M$. 
If the birational map is of the form $M_i\dashrightarrow M_j$, then the strict transform 
of $C\subset M_j$ on $M_i$ is denoted by $C^i$. 

Let $S$ be a nonsingular surface and let $D=\sum a_jD_j$ be an effective divisor on $S$ 
$(a_j>0)$ such that $|D|$ is simple normal crossing and $D_i$, $D_j$ are intersected 
at most one point. (It is sufficient in our situation.) 
The \emph{dual graph} of $D$ is defined as follows. A vertex corresponds 
to a component $D_j$. Let $v_j$ be the vertex corresponds to $D_j$. The 
$v_i$ and $v_j$ are joined by a (simple) line if and only if $D_i$, $D_j$ are intersected. 
In the dual graphs of divisors, a vertex corresponding to $(-n)$-curve is expressed 
as \textcircled{\tiny $n$}. 
On the other hand, an arbitrary irreducible curve is expressed by the symbol 
{\large$\oslash$} when it is not necessary a $(-n)$-curve.

Let $\F_n\to\pr^1$ be a Hirzebruch surface $\pr_{\pr^1}(\sO\oplus\sO(n))$ of 
degree $n$ with the $\pr^1$-fibration. 
A section $\sigma\subset\F_n$ with $(\sigma^2)=-n$ 
is called a \emph{minimal section}. If $n>0$, then such $\sigma$ is unique. 
We usually denote a fiber of $\F_n\to\pr^1$ by $l$. 

For a real number $t$, let $\lfloor t\rfloor$ be the greatest integer not grater than $t$. 
\end{nott}

\section{Elimination of subschemes}\label{elim_section}

In this section, we recall the results in \cite[\S 2]{N} (see also \cite[\S 2]{F}). 
Let $X$ be a nonsingular surface and $\Delta$ be a zero-dimensional subscheme 
of $X$. The defining ideal sheaf of $\Delta$ is denoted by $\sI_{\Delta}$.

\begin{definition}\label{nu1dfn}
Let $P$ be a point of $\Delta$. 
\begin{enumerate}
\renewcommand{\theenumi}{\arabic{enumi}}
\renewcommand{\labelenumi}{(\theenumi)}
\item\label{nu1dfn1}
Let 
$\nu_P(\Delta):=\max\{\nu\in\Z_{>0} \, | \, \sI_{\Delta}\subset\dm_P^\nu\}$,
where $\dm_P$ is the maximal ideal sheaf in $\sO_X$ defining $P$.
If $\nu_P(\Delta)=1$ for any $P\in\Delta$, then we say that $\Delta$ \emph{satisfies 
the $(\nu1)$-condition}.
\item\label{nu1dfn2}
The \emph{multiplicity} $\mult_P\Delta$ of $\Delta$ at $P$ is given by 
the length of the Artinian local ring $\sO_{\Delta, P}$.
\item\label{nu1dfn3}
The \emph{degree} $\deg\Delta$ of $\Delta$ is 
given by $\sum_{P\in\Delta}\mult_P\Delta$.
\end{enumerate}
\end{definition}

\begin{definition}\label{elimdfn}
Assume that $\Delta$ satisfies the $(\nu1)$-condition. 
Let $V\to X$ be the blowing up along $\Delta$. 
The \emph{elimination} of $\Delta$ is the birational projective morphism 
$\psi\colon Y\to X$ which is defined as the composition of the minimal resolution 
$Y\to V$ of $V$ and the morphism $V\to X$. 
For any divisor $E$ on $X$ and for any positive integer $s$, we set 
$E^{\Delta,s}:=\psi^*E-sK_{Y/X}$. 
\end{definition}

\begin{proposition}[{\cite[Proposition 2.9]{N}}]\label{Nelprop}
\begin{enumerate}
\renewcommand{\theenumi}{\arabic{enumi}}
\renewcommand{\labelenumi}{(\theenumi)}
\item\label{Nelprop1}
Assume that the subscheme $\Delta$ satisfies the $(\nu1)$-condition 
and let $\psi\colon Y\to X$ 
be the elimination of $\Delta$. 
Then the anti-canonical divisor $-K_Y$ is $\psi$-nef. 
More precisely, for any $P\in\Delta$ with $\mult_P\Delta=k$, 
the set-theoretic inverse image $\psi^{-1}(P)$ is the 
straight chain $\sum_{j=1}^k\Gamma_{P,j}$ of 
nonsingular rational curves and the dual graph of $\psi^{-1}(P)$ is the following: 
\begin{center}
    \begin{picture}(150, 50)(0, 45)
    \put(-6, 57){\textcircled{\tiny $2$}}
    \put(0, 70){\makebox(0, 0)[b]{$\Gamma_{P,1}$}}
    \put(5, 60){\line(1, 0){30}}
    \put(34, 57){\textcircled{\tiny $2$}}
    \put(40, 70){\makebox(0, 0)[b]{$\Gamma_{P,2}$}}
    \put(45, 60){\line(1, 0){20}} 
    \put(67, 60){\line(1, 0){2}}
    \put(71, 60){\line(1, 0){2}}
    \put(75, 60){\line(1, 0){2}}    
    \put(79, 60){\line(1, 0){21}}
    \put(99, 57){\textcircled{\tiny $2$}}
    \put(105, 70){\makebox(0, 0)[b]{$\Gamma_{P,k-1}$}}
    \put(110, 60){\line(1, 0){30}}
    \put(139, 57){\textcircled{\tiny $1$}}
    \put(145, 70){\makebox(0, 0)[b]{$\Gamma_{P,k}$}}
    \end{picture}
\end{center}
\item\label{Nelprop2}
Conversely, for a proper birational morphism $\psi\colon Y\to X$
between nonsingular surfaces such that $-K_Y$ is $\psi$-nef, 
the morphism $\psi$ 
is the elimination of $\Delta$ which satisfies the $(\nu1)$-condition 
defined by the ideal $\sI_\Delta:=\psi_*\sO_Y(-K_{Y/X})$.
\end{enumerate}
\end{proposition}

\begin{definition}\label{Gammadfn}
Under the assumption of Proposition \ref{Nelprop} \eqref{Nelprop1}, we always denote 
the exceptional curves of $\psi$ over $P$ by $\Gamma_{P, 1},\dots,\Gamma_{P, k}$. 
The order is determined as Proposition \ref{Nelprop} \eqref{Nelprop1}. 
\end{definition}

Now we see some examples of the dual graphs of $E^{\Delta, s}$.

\begin{example}[{\cite[Example 2.5]{F}}]\label{E1}
Assume that $\Delta$ satisfies the $(\nu1)$-condition such that $|\Delta|=\{P\}$.
Let $E=eC$ be a divisor on $X$ such that $P\in C$ and $C$ is 
nonsingular. Let $m:=\deg\Delta$ and $k:=\mult_P(\Delta\cap C)$. 
Then we have 
\[
E^{\Delta, s}=eC^Y+\sum_{i=1}^ki(e-s)\Gamma_{P,i}+\sum_{i=k+1}^m(ek-si)\Gamma_{P,i},
\]
where $\psi\colon Y\to X$ is the elimination of $\Delta$. 
Moreover, the dual graph of $\psi^{-1}(E)$ is the following: 
\begin{center}
    \begin{picture}(180, 70)(0, 10)
    \put(-6, 57){\textcircled{\tiny $2$}}
    \put(0, 70){\makebox(0, 0)[b]{$\Gamma_{P,1}$}}
    \put(5, 60){\line(1, 0){20}} 
    \put(27, 60){\line(1, 0){2}}
    \put(31, 60){\line(1, 0){2}}
    \put(35, 60){\line(1, 0){2}}    
    \put(39, 60){\line(1, 0){21}}
    \put(59, 57){\textcircled{\tiny $2$}}
    \put(65, 70){\makebox(0, 0)[b]{$\Gamma_{P,k}$}}
    \put(70, 60){\line(1, 0){20}} 
    \put(92, 60){\line(1, 0){2}}
    \put(96, 60){\line(1, 0){2}}
    \put(100, 60){\line(1, 0){2}}    
    \put(104, 60){\line(1, 0){21}}
    \put(124, 57){\textcircled{\tiny $2$}}
    \put(130, 70){\makebox(0, 0)[b]{$\Gamma_{P,m-1}$}}
    \put(135, 60){\line(1, 0){30}}
    \put(164, 57){\textcircled{\tiny $1$}}
    \put(170, 70){\makebox(0, 0)[b]{$\Gamma_{P,m}$}}
    \put(65, 54.5){\line(0, -1){24.95}}
    \put(65, 25){\makebox(0, 0){\large$\oslash$}}
    \put(85, 18){\makebox(0, 0)[b]{$C^Y$}}
    \end{picture}
\end{center}
\end{example}

\begin{example}{\cite[Example 2.6]{F}}\label{E2}
Assume that $\Delta$ satisfies the $(\nu1)$-condition such that $|\Delta|=\{P\}$.
Let $E=e_1C_1+e_2C_2$ be a non-zero effective divisor such that $C_1$ and $C_2$ 
are nonsingular and intersect transversally at a unique point $P=C_1\cap C_2$. 
Let $m:=\deg\Delta$ and $k_j:=\mult_P(\Delta\cap C_j)$. 
By \cite[Lemma 2.12]{N}, we may assume that $k_1=1$.  
Then we have
\[
E^{\Delta, s}=e_1C_1^Y+e_2C_2^Y+
\sum_{i=1}^{k_2}(i(e_2-s)+e_1)\Gamma_{P,i}+\sum_{i=k_2+1}^m(e_1+k_2e_2-is)\Gamma_{P,i},
\]
where $\psi\colon Y\to X$ is the elimination of $\Delta$. 
Moreover, the dual graph of $\psi^{-1}(E)$ is the following: 
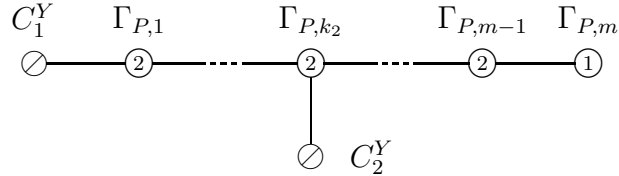
\begin{center}
    \begin{picture}(220, 70)(0, 20)
    \put(0.5, 60){\makebox(0, 0){\large$\oslash$}}
    \put(0, 70){\makebox(0, 0)[b]{$C_1^Y$}}
    \put(5, 60){\line(1, 0){30}}
    \put(34, 57){\textcircled{\tiny $2$}}
    \put(40, 70){\makebox(0, 0)[b]{$\Gamma_{P,1}$}}
    \put(45, 60){\line(1, 0){20}} 
    \put(67, 60){\line(1, 0){2}}
    \put(71, 60){\line(1, 0){2}}
    \put(75, 60){\line(1, 0){2}}    
    \put(79, 60){\line(1, 0){21}}
    \put(99, 57){\textcircled{\tiny $2$}}
    \put(105, 70){\makebox(0, 0)[b]{$\Gamma_{P,k_2}$}}
    \put(110, 60){\line(1, 0){20}} 
    \put(132, 60){\line(1, 0){2}}
    \put(136, 60){\line(1, 0){2}}
    \put(140, 60){\line(1, 0){2}}    
    \put(144, 60){\line(1, 0){21}}
    \put(164, 57){\textcircled{\tiny $2$}}
    \put(170, 70){\makebox(0, 0)[b]{$\Gamma_{P,m-1}$}}
    \put(175, 60){\line(1, 0){30}}
    \put(204, 57){\textcircled{\tiny $1$}}
    \put(210, 70){\makebox(0, 0)[b]{$\Gamma_{P,m}$}}
    \put(105, 54.8){\line(0, -1){25}}
    \put(105, 25){\makebox(0, 0){\large$\oslash$}}
    \put(128, 18){\makebox(0, 0)[b]{$C_2^Y$}}
    \end{picture}
\end{center}
\end{example}

\section{Log del Pezzo surfaces}\label{dP_section}

We define the notion of log del Pezzo surfaces, 
$a$-basic pairs, and $a$-fundamental multiplets, 
and we see the correspondence among them.

\subsection{Log del Pezzo surfaces and $a$-basic pairs}\label{dPb_section}

\begin{definition}\label{dPdfn}
\begin{enumerate}
\renewcommand{\theenumi}{\arabic{enumi}}
\renewcommand{\labelenumi}{(\theenumi)}
\item\label{dPdfn1}
A normal projective surface $S$ is called a \emph{log del Pezzo surface} 
if $S$ is log-terminal and the anti-canonical divisor $-K_S$ is an ample $\Q$-Cartier 
divisor.
\item\label{dPdfn2}
Let $S$ be a log del Pezzo surface. The \emph{index} of $S$ is defined as 
$\min\{a\in\Z_{>0} | -aK_S\text{ is Cartier}\}$.
\end{enumerate}
\end{definition}

\begin{remark}\label{dPrmk}
Any log del Pezzo surface is a rational surface by \cite[Proposition 3.6]{N}. 
In particular, the Picard group $\Pic(S)$ of $S$ is a finitely generated and torsion-free 
Abelian group. 
\end{remark}

\begin{definition}[{\cite[Definition 3.3]{FY}}]\label{basdfn}
Fix $a\geq 2$. 
A pair $(M, E_M)$ is called an \emph{$a$-basic pair} if the following conditions are 
satisfied: 
\begin{enumerate}
\renewcommand{\theenumi}{\arabic{enumi}}
\renewcommand{\labelenumi}{($\sC$\theenumi)}
\item\label{basdfn1}
$M$ is a nonsingular projective rational surface.
\item\label{basdfn2}
$E_M$ is a nonzero effective divisor on $M$ such that 
$\coeff E_M\subset\{1,\dots,a-1\}$ 
and $|E_M|$ is simple normal crossing.
\item\label{basdfn3}
A divisor $L_M\sim -aK_M-E_M$ (called the \emph{fundamental divisor} of 
$(M, E_M)$) satisfies that $K_M+L_M$ is nef and $(K_M+L_M\cdot L_M)>0$. 
\item\label{basic_dfn4}
For any irreducible component $C\leq E_M$, $(L_M\cdot C)=0$ holds. 
\end{enumerate}
\end{definition}

We see the correspondence between log del Pezzo surfaces and $a$-basic pairs.

\begin{proposition}[{\cite[Proposition 3.4]{FY}}]\label{dPbprop}
Fix $a\geq 2$. 
\begin{enumerate}
\renewcommand{\theenumi}{\arabic{enumi}}
\renewcommand{\labelenumi}{(\theenumi)}
\item\label{dPbprop1}
Let $S$ be a non-Gorenstein log del Pezzo surface such that $-aK_S$ is Cartier. 
Let $\alpha\colon M\to S$ be the minimal resolution 
of $S$ and let $E_M:=-aK_{M/S}$. 
Then $(M, E_M)$ is an $a$-basic pair and the divisor $\alpha^*(-aK_S)$ is the 
fundamental divisor of $(M, E_M)$. 
\item\label{dPbprop2}
Let $(M, E_M)$ be an $a$-basic pair and $L_M$ be 
the fundamental divisor of $(M, E_M)$.
Then there exists a projective and birational morphism $\alpha\colon M\to S$ 
such that $S$ is a non-Gorenstein log del Pezzo surface with $-aK_S$ Cartier and 
$L_M\sim\alpha^*(-aK_S)$ holds. 
Moreover, the morphism $\alpha$ is the minimal resolution of $S$. 
\end{enumerate}
In particular, under the correspondence between $S$ and $(M, E_M)$, 
we have $a(-K_S^2)=(1/a)(L_M^2)$, where $L_M$ is the fundamental divisor.
\end{proposition}

\begin{definition}\label{associate_dfn}
Let $a\geq 2$. 
\begin{enumerate}
\renewcommand{\theenumi}{\arabic{enumi}}
\renewcommand{\labelenumi}{(\theenumi)}
\item\label{associate_dfn1}
For a log del Pezzo surface $S$ of index $a$, the corresponding 
$a$-basic pair $(M, E_M)$ which is given in 
Proposition \ref{dPbprop} \eqref{dPbprop1} is called 
the \emph{associated $a$-basic pair} of $S$. 
\item\label{associate_dfn1}
For an $a$-basic pair $(M, E_M)$, the corresponding log del Pezzo surface $S$ 
which is given in Proposition \ref{dPbprop} \eqref{dPbprop2} 
is called the \emph{associated log del Pezzo surface} of $(M, E_M)$. 
We note that $a$ is divisible by the index of $S$. 
\end{enumerate}
\end{definition}

We discuss that when the log del Pezzo surface $S$ associated to an $a$-basic pair 
$(M, E_M)$ is of index $a$.

\begin{lemma}\label{2alem}
Let $S$ be a log del Pezzo surface of index $a\geq 2$ and $(M, E_M)$ be the 
associated $a$-basic pair. Pick any irreducible component $C\leq E_M$ and set 
$e:=\coeff_CE_M$ and $d:=-(C^2)$. Then $2\leq d\leq 2a/(a-e)$. 
In particular, $d\leq 2a$.
\end{lemma}

\begin{proof}
By Proposition \ref{dPbprop}, $d\geq 2$. We know that 
$-ed\leq (E_M\cdot C)=(-aK_M\cdot C)=a(2-d)$ since $(L_M\cdot C)=0$, 
where $L_M$ is the fundamental divisor. 
Thus $d\leq 2a/(a-e)$. Since $e\leq a-1$, we have $d\leq 2a$. 
\end{proof}

\begin{corollary}\label{indacor}
Let $(M, E_M)$ be an $a$-basic pair with $a\geq 2$. Assume that there exists an 
irreducible component $C\leq E_M$ such that either $(C^2)<-a$, or 
$\coeff_CE_M$ and $a$ are coprime. Then the associated log del Pezzo surface $S$ 
is of index $a$. 
\end{corollary}

\begin{proof}
Let $a'$ be the index of $S$. Then $a'>1$ and $a$ is divisible by $a'$. Moreover, 
$\coeff_CE_M$ is divisible by $a/a'$. Thus the assertion follows from Lemma \ref{2alem}. 
\end{proof}

\subsection{$a$-fundamental multiplets}\label{mult_section}

In order to determine $a$-basic pairs, we define the notion of 
$a$-fundamental multiplets given in \cite[\S 1]{FY}.

\begin{definition}\label{funddfn}
Fix $a\geq 2$ and $1\leq i\leq a-1$. We define the following notion inductively. 
A multiplet $(M_i, E_i; \Delta_1,\dots,\Delta_i)$ is called an 
\emph{$a$-pseudo-fundamental multiplet of length $i$} if the following conditions are 
satisfied: 
\begin{enumerate}
\renewcommand{\theenumi}{\arabic{enumi}}
\renewcommand{\labelenumi}{($\sF$\theenumi)}
\item\label{funddfn1}
$M_i$ is a nonsingular projective rational surface 
and $E_i$ is a nonzero effective divisor on $M_i$. 
\item\label{funddfn2}
A divisor $L_i\sim-aK_{M_i}-E_i$ $($called the 
\emph{fundamental divisor}$)$ satisfies that 
$iK_{M_i}+L_i$ is nef and $((i+1)K_{M_i}+L_i\cdot\gamma)\geq 0$ for any $(-1)$-curve 
$\gamma\subset M_i$. 
\item\label{funddfn3}
$\Delta_i\subset M_i$ is a zero-dimensional subscheme which satisfies 
the $(\nu1)$-condition. 
\item\label{funddfn4}
Let $\pi_i\colon M_{i-1}\to M_i$ be the elimination of $\Delta_i$ and 
$E_{i-1}:=E_i^{\Delta_i, a-i}$. Then the following holds: 
\begin{itemize}
\item
If $i=1$, then $(M_0, E_0)$ is an $a$-basic pair. 
\item
If $i\geq 2$, then $(M_{i-1}, E_{i-1}; \Delta_1,\dots,\Delta_{i-1})$ is an 
$a$-pseudo-fundamental multiplet of length $i-1$. 
\end{itemize}
\end{enumerate}
Moreover, if $(i+1)K_{M_i}+L_i$ is not nef, then we call the multiplet an 
\emph{$a$-fundamental multiplet} of length $i$. 
\end{definition}

\begin{definition}\label{fundrmk_dfn}
\begin{enumerate}
\renewcommand{\theenumi}{\arabic{enumi}}
\renewcommand{\labelenumi}{(\theenumi)}
\item\label{fundrmk_dfn1}
Let $(M_i, E_i; \Delta_1,\dots,\Delta_i)$ be an $a$-pseudo-funda-mental multiplet. 
We call the $a$-basic pair $(M_0, E_0)$ (constructed from the multiplet inductively) 
as the \emph{associated $a$-basic pair}. 
\item\label{fundrmk_dfn2}
We sometimes call $a$-basic pairs as \emph{$a$-pseudo-fundamental multiplets 
of length zero} for convenience. 
\end{enumerate}
\end{definition}

The following proposition is important.

\begin{proposition}\label{fundprop}
Fix $a\geq 2$ and $0\leq i\leq a-1$. 
Let $(M_i, E_i; \Delta_1,\dots,\Delta_i)$ be an $a$-pseudo-fundamental multiplet of 
length $i$, $\pi_j\colon M_{j-1}\to M_j$ be the elimination of $\Delta_j$, 
$E_{j-1}:=E_j^{\Delta_j, a-j}$ and $L_j$ be the fundamental divisor of 
$(M_j, E_j; \Delta_1,\dots,\Delta_j)$ for any $0\leq j\leq i$.
\begin{enumerate}
\renewcommand{\theenumi}{\arabic{enumi}}
\renewcommand{\labelenumi}{(\theenumi)}
\item\label{fundprop1}
Assume that $(i+1)K_{M_i}+L_i$ is nef. Then $i\leq a-2$ and $iK_{M_i}+L_i$ is nef and big. 
Moreover, there exists a projective and birational morphism 
$\pi_{i+1}\colon M_i\to M_{i+1}$ between nonsingular surfaces 
such that the following conditions are satisfied: 
\begin{itemize}
\item
There exists a zero-dimensional subscheme $\Delta_{i+1}\subset M_{i+1}$ which 
satisfies the $(\nu1)$-condition such that $\pi_{i+1}$ is the elimination of $\Delta_{i+1}$. 
\item
Set $E_{i+1}:=(\pi_{i+1})_*E_i$ and $L_{i+1}:=(\pi_{i+1})_*L_i$. Then 
$E_i=E_{i+1}^{\Delta_{i+1}, a-i-1}$ and $L_i=L_{i+1}^{\Delta_{i+1}, i+1}$. 
\item
$(M_{i+1}, E_{i+1}; \Delta_1,\dots,\Delta_{i+1})$ is an $a$-pseudo-fundamental 
multiplet of length $i+1$ and $L_{i+1}$ is the fundamental divisor. 
\end{itemize}
\item\label{fundprop2}
Assume that $(i+1)K_{M_i}+L_i$ is not nef, that is, the multiplet 
$(M_i, E_i; \Delta_1,\dots,\Delta_i)$ is an $a$-fundamental multiplet. 
Then $M_i$ is isomorphic to either $\pr^2$ or $\F_n$, and 
$((i+1)K_{M_i}+L_i\cdot l)<0$, where $l$ is a line $($if $M_i\simeq\pr^2)$; a fiber 
$($if $M_i\simeq\F_n)$. 
\item\label{fundprop3}
$L_i$ is nef and big. Moreover, we have the following: 
\begin{itemize}
\item
$(L_i\cdot E_i)=\sum_{j=1}^ij(a-j)\deg\Delta_j$.
\item
$(K_{M_i}+L_i\cdot L_i)-(K_{M_0}+L_0\cdot L_0)=\sum_{j=1}^ij(j-1)\deg\Delta_j$.
\item
$(L_i\cdot C)=\sum_{j=1}^ij\deg(\Delta_j\cap C^j)$ for any nonsingular component 
$C\leq E_i$.
\item
$(1/a)(L_0^2)=(-K_{M_i}\cdot L_i)-\sum_{j=1}^ij\deg\Delta_j$. 
\end{itemize}
\end{enumerate}
\end{proposition}

\begin{proof}
We can assume by induction on $i$ 
that $L_{i-1}$ is nef and big and $L_{i-1}=L_i^{\Delta_i, i}$ if $i\geq 1$. 
Thus $L_i$ is nef and big. 

\eqref{fundprop1}
Assume that $(i+1)K_{M_i}+L_i$ is nef. Since 
$a((i+1)K_{M_i}+L_i)\sim-(i+1)E_i+(a-(i+1))L_i$, we must have $a-(i+1)>0$. Moreover, 
since $(i+1)(iK_{M_i}+L_i)=i((i+1)K_{M_i}+L_i)+L_i$, $iK_{M_i}+L_i$ is nef and big. 
From now on, we run $((i+2)K_{M_i}+L_i)$-minimal model program and let 
$\pi_{i+1}\colon M_i\to M_{i+1}$ be the composition of the morphisms in the program. 
More precisely, we obtain the morphism $\pi_{i+1}$ which is the composition of 
monoidal transforms such that each exceptional $(-1)$-curve intersects 
the strict transform 
of $(i+2)K_{M_i}+L_i$ negatively. Furthermore, 
$((i+2)K_{M_{i+1}}+L_{i+1}\cdot \gamma)\geq 0$ holds for any $(-1)$-curve 
$\gamma\subset M_{i+1}$. Since $(i+1)K_{M_i}+L_i$ is nef, any $(-1)$-curve in each 
step of the monoidal transform intersects the strict transform 
of $(i+1)K_{M_i}+L_i$ trivially. Thus $(i+1)K_{M_i}+L_i=\pi_{i+1}^*((i+1)K_{M_{i+1}}+L_{i+1})$. 
In particular, $-K_{M_i}$ is $\pi_{i+1}$-nef. By Proposition \ref{Nelprop}, there exists 
a zero-dimensional subscheme $\Delta_{i+1}\subset M_{i+1}$ which satisfies the 
$(\nu1)$-condition such that $\pi_{i+1}$ is the elimination of $\Delta_{i+1}$. 
Furthermore, we have $L_i=L_{i+1}^{\Delta_{i+1}, i+1}$ and 
$E_i=E_{i+1}^{\Delta_{i+1}, a-i-1}$. 
Since each step of the $((i+2)K_{M_i}+L_i)$-minimal model program can be seen 
as a step of $(-E_i)$-minimal model program, $E_{i+1}$ is nonzero effective. 
Thus the assertion \eqref{fundprop1} follows. 

\eqref{fundprop2}
Follows from \cite[Theorem 2.1]{mori}. 

\eqref{fundprop3}
We know that $(-K_{M_i}\cdot L_i)=(-K_{M_{i-1}}\cdot L_{i-1})-i(K_{M_{i-1}/K_{M_i}}^2)
=(-K_{M_{i-1}}\cdot L_{i-1})+i\deg\Delta_i$. 
Thus we have $(1/a)(L_0^2)=(-K_{M_i}\cdot L_i)-\sum_{j=1}^ij\deg\Delta_j$ by 
induction on $i$. 
Other assertions follow similarly. 
\end{proof}

As a direct corollary of Proposition \ref{fundprop} \eqref{fundprop1}, 
we have the following.

\begin{corollary}\label{bfcor}
Fix $a\geq 2$. Let $(M, E_M)$ be an $a$-basic pair. Then there exists 
a positive integer $1\leq b\leq a-1$ 
and an $a$-fundamental multiplet $(M_b, E_b; \Delta_1,\dots,\Delta_b)$ of length $b$ 
such that the associated $a$-basic pair is isomorphic to $(M, E_M)$. 
\end{corollary}

From the next proposition, we can replace an $a$-fundamental multiplet with 
another one. The proof is same as that of \cite[Proposition 4.4]{N}. 
See also \cite[Proposition 3.14]{F} and \cite[Theorem 3.12]{FY}.

\begin{proposition}\label{hirzprop}
Fix $a\geq 2$ and $1\leq b< a$. Let $(M_b, E_b;\Delta_1,\dots,\Delta_b)$ be 
an $a$-fundamental multiplet of length $b$ and $L_b$ be the fundamental divisor.
Assume that $M_b=\F_n$ and $bK_{M_b}+L_b$ is not big and not trivial. Then 
there exists an $a$-fundamental multiplet 
$(M'_b, E'_b; \Delta_1,\dots,\Delta_{b-1}, \Delta'_b)$ such that 
the associated $a$-pseudo-fundamental multiplets of length $b-1$ are same, 
$M'_b=\F_{n'}$ for some $n'\geq 0$ and $\Delta'_b\cap\sigma'=\emptyset$, where 
$\sigma'\subset\F_{n'}$ is a section with $(\sigma'^2)=-n'$. 
\end{proposition}

The next proposition is useful to know when a given multiplet is an 
$a$-pseudo-fundamental multiplet.

\begin{proposition}\label{convprop}
Fix $1\leq i<a$. Let $X$ be a nonsingular projective surface, $E$ be a divisor on $X$, 
$L\sim-aK_X-E$ be a divisor such that $iK_X+L$ is nef, $\Delta\subset X$ be a 
zero-dimensional subscheme which satisfies the $(\nu1)$-condition, 
$\psi\colon Y\to X$ be the elimination of $\Delta$, $L_Y:=L^{\Delta, i}$ and 
$E_Y:=E^{\Delta, a-i}$. If $E_Y$ is effective and $(L_Y\cdot C)\geq 0$ 
for any irreducible 
component $C\leq E_Y$, then $jK_Y+L_Y$ is nef for any $0\leq j\leq i$. 
\end{proposition}

\begin{proof}
Assume that there exists an integer $0\leq j\leq i$ such that $jK_Y+L_Y$ is non-nef. 
Since $iK_Y+L_Y=\psi^*(iK_X+L)$, we have $j<i$. 
Pick an irreducible curve $B\subset Y$ such that $(jK_Y+L_Y\cdot B)<0$. Then 
\begin{eqnarray*}
0 & > & (a-i)(jK_Y+L_Y\cdot B)\\
 & = & (a-j)(iK_Y+L_Y\cdot B)+(i-j)(E_Y\cdot B)\geq(i-j)(E_Y\cdot B). 
\end{eqnarray*}
Thus $B\leq E_Y$. Hence $(L_Y\cdot B)\geq 0$. However, 
\[
0>i(jK_Y+L_Y\cdot B)=j(iK_Y+L_Y\cdot B)+(i-j)(L_Y\cdot B)\geq 0,
\]
which leads to a contradiction. 
\end{proof}

\subsection{Local properties}

In this section, we fix the following notation: 
Let $1\leq i<a$, $(M_i, E_i; \Delta_1,\dots,\Delta_i)$ be an $a$-pseudo-fundamental multiplet of length $i$, $L_i$ be its fundamental divisor, 
$\pi_j\colon M_{j-1}\to M_j$ be the elimination of $\Delta_j$, 
$E_{j-1}:=E_j^{\Delta_j, a-j}$ and $L_{j-1}:=L_j^{\Delta_j, j}$ for $1\leq j\leq i$. 
Moreover, we fix a point $P\in\Delta_i$.

\begin{lemma}\label{FS1}
\begin{enumerate}
\renewcommand{\theenumi}{\arabic{enumi}}
\renewcommand{\labelenumi}{(\theenumi)}
\item\label{FS11}
The multiplicity of $E_i$ at $P$ is bigger than or equal to $a-i$. 
\item\label{FS12}
Pick any $\pi_i$-exceptional curve $\Gamma\subset M_{i-1}$. 
If $(\Gamma^2)=-1$, then $(L_{i-1}\cdot\Gamma)=i$. 
If $(\Gamma^2)=-2$, then $(L_{i-1}\cdot\Gamma)=0$. 
\item\label{FS13}
Assume that $E_i=eC$ around $P$, where $P\in C$ and $C$ is nonsingular and 
$e\leq a-i$. Then $e=a-i$, $\Delta_i\subset C$ around $P$ and $E_{i-1}$ is the 
strict transform of $E_i$ around over $P$. 
\item\label{FS14}
Assume that $E_i=(a-1)C$ around $P$, where $P\in C$ and $C$ is nonsingular. If 
$2i\leq a+1$, $\mult_P(\Delta_i\cap C)=1$ and $\mult_P\Delta_i\geq 2$, 
then $2i=a+1$ and $\mult_P\Delta_i=2$. 
\end{enumerate}
\end{lemma}

\begin{proof}
\eqref{FS11} 
The value $\coeff_{\Gamma_{P, 1}}E_{i-1}$ must be nonnegative. 

\eqref{FS12} 
Follows from the fact $(iK_{M_{i-1}}+L_{i-1}\cdot\Gamma)=0$. 

\eqref{FS13}
Set $m:=\mult_P\Delta_i$ and $k:=\mult_P(\Delta_i\cap C)$. 
By Example \ref{E1}, $\coeff_{\Gamma_{P,1}}E_{i-1}=e-(a-i)$. Thus $e=a-i$. If $k<m$, 
then $\coeff_{\Gamma_{P,m}}E_{i-1}=(a-i)(k-m)<0$ by Example \ref{E1}, 
a contradiction. 

\eqref{FS14}
We know that $\coeff_{\Gamma_{P, 2}}E_{i-1}=a-1-2(a-i)\leq 0$ by Example \ref{E1}. 
Thus $2i=a+1$ and $\mult_P\Delta_i=2$. 
\end{proof}

\begin{lemma}\label{FS2}
Assume that $a\geq 4$, $i=1$, $1\leq e\leq 2$, $E_1=(a-1)C_1+eC_2$ around $P$, 
where $C_1$, $C_2$ are nonsingular and intersect transversally at $P$. 
Then $e=2$ and $\mult_P\Delta_1=\mult_P(\Delta_1\cap C_2)$. Moreover, 
$(a, \mult_P\Delta_1)=(5,2)$ or $(4,3)$. 
\end{lemma}

\begin{proof}
Set $m:=\mult_P\Delta_1$ and $k_j:=\mult_P(\Delta_1\cap C_j)$. 
By Lemma \ref{2alem}, $\coeff_{\Gamma_{P,m}}E_0=0$. 
If $k_1>1$, then $k_2=1$ by Example \ref{E2}. However, if $m=k_1$, then 
$\coeff_{\Gamma_{P, m}}E_0=e$; if $m>k_1$, then 
$\coeff_{\Gamma_{P, k_1+1}}E_0=e-(a-1)<0$. This leads to a contradiction. 
Thus $k_1=1$. If $m>k_2$, then 
$\coeff_{\Gamma_{P, m}}E_0=a-1+ek_2-m(a-1)\leq 2k_2-k_2(a-1)<0$ 
by Example \ref{E2}. This leads to a contradiction. Thus $m=k_2$. 
Again by Example \ref{E2}, $\coeff_{\Gamma_{P, m}}E_0=m(e+1-a)+a-1$. Thus 
$m=(a-1)/(a-(e+1))$. Thus $e=2$. Moreover, $(a,m)=(5,2)$ or $(4,3)$. 
\end{proof}

\begin{lemma}\label{FS3}
Assume that $i\geq 2$, $\Delta_j=\emptyset$ for any $1\leq j<i$, 
$E_i=(a-i+1)C$ around $P$ and $\mult_P\Delta_i<2(a-i+1)$, where $P\in C$ and 
$C$ is nonsingular. Then $\mult_P\Delta_i=a-i+1$ and $\mult_P(\Delta_i\cap C)=a-i$. 
\end{lemma}

\begin{proof}
Set $m:=\mult_P\Delta_i$ and $k:=\mult_P(\Delta_i\cap C)$. 
By Lemma \ref{2alem}, $E_{i-1}(=E_0)$ does not contain any $(-1)$-curve. 
Thus $0=\coeff_{\Gamma_{P, m}}E_{i-1}=(a-i+1)k-(a-i)m$ by Example \ref{E1}. 
Thus the assertion follows. 
\end{proof}

\section{Examples}\label{ex_section}

In this section, we see the $a$-fundamental multiplets and the log del Pezzo surfaces 
which appeared in Section \ref{intro_section}. 

\subsection{Special $a$-fundamental multiplets}\label{ex_section1}

\begin{example}\label{fEO}
Let $a\geq 2$ and $b:=\lfloor(a+1)/2\rfloor$. We consider the following 
$a$-fundamental multiplets $(M_b, E_b; \Delta_1,\dots,\Delta_b)$ of length $b$: 
\begin{description}
\item[{$\langle{\rm O}\rangle_a$}]
$M_b=\F_{2a}$, $E_b=(a-1)\sigma$, $\Delta_b=\emptyset,\dots,\Delta_1=\emptyset$. 
\item[{$\langle{\rm I}\rangle_a$}]
$M_b=\F_{2a-1}$, $E_b=(a-1)\sigma$, $\Delta_b=\emptyset,\dots,\Delta_2=\emptyset$, 
$\Delta_1\subset\sigma^1$ with $\deg\Delta_1=1$. 
\item[{$\langle{\rm II}\rangle_a$}]
$M_b=\F_{2a-2}$, $E_b=(a-1)\sigma$, $\Delta_b=\emptyset,\dots,\Delta_2=\emptyset$, 
$\Delta_1\subset\sigma^1$ with $\deg\Delta_1=2$. 
\item[{$\langle{\rm III}\rangle_a$}]
$M_b=\F_{2a-3}$, $E_b=(a-1)\sigma$, $\Delta_b=\emptyset,\dots,\Delta_2=\emptyset$, 
$\Delta_1\subset\sigma^1$ with $\deg\Delta_1=3$. 
\item[{$\langle{\rm IV}\rangle_a$}]
$M_b=\F_{2a-4}$, $E_b=(a-1)\sigma$, $\Delta_b=\emptyset,\dots,\Delta_2=\emptyset$, 
$\Delta_1\subset\sigma^1$ with $\deg\Delta_1=4$. 
\end{description}
By Proposition \ref{convprop}, it is easy to check that these multiplets are exactly 
$a$-fundamental multiplets. Moreover, all of the associated $a$-basic pairs $(M_0, E_0)$ 
satisfy that $E_0=(a-1)\sigma^0$ and $((\sigma^0)^2)=-2a$. Thus the associated 
log del Pezzo surfaces $S$ are of index $a$ by Corollary \ref{indacor}. 
Furthermore, the value $(-K_S^2)$ is equal to $(2a^2+4a+2)/a$ 
(if $\langle$O$\rangle_a$), $(2a^2+3a+2)/a$ (if $\langle$I$\rangle_a$), 
$(2a^2+2a+2)/a$ (if $\langle$II$\rangle_a$), $(2a^2+a+2)/a$ (if $\langle$III$\rangle_a$), 
$(2a^2+2)/a$ (if $\langle$IV$\rangle_a$) by Proposition \ref{fundprop} \eqref{fundprop3}. 
\end{example}

\begin{example}\label{fEA}
We consider the case $(a, b)=(5,3)$. 
Consider the following $5$-fundamental multiplet $(M_3, E_3; \Delta_1,\Delta_2, \Delta_3)$ of length three: 
\begin{description}
\item[{$\langle{\rm A}\rangle_5$}]
$M_3=\F_8$, $E_3=4\sigma+2l$, $\Delta_3\subset l\setminus\sigma$ with 
$\deg\Delta_3=2$, $\Delta_2=\emptyset$ and $\Delta_1=\emptyset$. 
\end{description}
By Proposition \ref{convprop}, it is easy to check that the multiplet is exactly 
a $5$-fundamental multiplet. Moreover, the associated $5$-basic pair $(M_0, E_0)$ 
satisfies that $E_0=4\sigma^0+2l^0$ and the dual graph of $E_0$ is the 
following: 
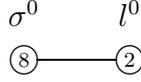
\begin{center}
    \begin{picture}(50, 40)(0, 50)
    \put(-6, 57){\textcircled{\tiny $8$}}
    \put(0, 73){\makebox(0, 0)[b]{$\sigma^0$}}
    \put(5, 60){\line(1, 0){30}}
    \put(34, 57){\textcircled{\tiny $2$}}
    \put(40, 73){\makebox(0, 0)[b]{$l^0$}}
    \end{picture}
\end{center}
Furthermore, the associated log del Pezzo surface $S$ of index five satisfies that 
the value $(-K_S^2)$ is equal to $54/5$ by Proposition \ref{fundprop} \eqref{fundprop3}. 
\end{example}

\begin{example}\label{fEB}
We consider the case $(a, b)=(4,2)$. 
Consider the following $4$-fundamental multiplets $(M_2, E_2; \Delta_1,\Delta_2)$ 
of length two: 
\begin{description}
\item[{$\langle{\rm B}\rangle_4$}]
$M_2=\F_4$, $E_2=3\sigma$, $|\Delta_2|=\{P\}$ with $P\in\sigma$ such that 
$\deg\Delta_2=3$, $\deg(\Delta_2\cap\sigma)=2$ and $\Delta_1=\emptyset$. 
\item[{$\langle{\rm C}\rangle_4$}]
$M_2=\F_5$, $E_2=3\sigma+2l$, $\deg\Delta_2=1$ with 
$\Delta_2\subset l\setminus\sigma$, 
$|\Delta_1|=\{P\}$ such that $P=\sigma^1\cap l^1$ and 
$\deg\Delta_1=\deg(\Delta_1\cap l^1)=3$. 
\end{description}
By Proposition \ref{convprop}, it is easy to check that these multiplets are exactly 
$4$-fundamental multiplets. Moreover, the associated $4$-basic pairs $(M_0, E_0)$ 
satisfy the following: 
\begin{description}
\item[{\rm The case} {$\langle{\rm B}\rangle_4$}]
$E_0=3\sigma^0+2\Gamma_{P,2}^0+\Gamma_{P,1}^0$ and the dual graph of $E_0$ is 
the following: 
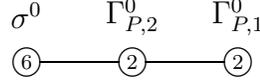
\begin{center}
    \begin{picture}(70, 40)(0, 50)
    \put(-6, 57){\textcircled{\tiny $6$}}
    \put(0, 73){\makebox(0, 0)[b]{$\sigma^0$}}
    \put(5, 60){\line(1, 0){30}}
    \put(34, 57){\textcircled{\tiny $2$}}
    \put(40, 70){\makebox(0, 0)[b]{$\Gamma_{P,2}^0$}}
    \put(45, 60){\line(1, 0){30}}
    \put(74, 57){\textcircled{\tiny $2$}}
    \put(80, 70){\makebox(0, 0)[b]{$\Gamma_{P,1}^0$}}
    \end{picture}
\end{center}
\item[{\rm The case} {$\langle{\rm C}\rangle_4$}]
$E_0=3\sigma^0+2\Gamma_{P,1}+\Gamma_{P,2}+2l^0$ 
and the dual graph of $E_0$ is the following: 
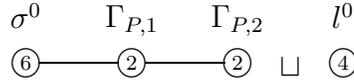
\begin{center}
    \begin{picture}(100, 40)(0, 50)
    \put(-6, 57){\textcircled{\tiny $6$}}
    \put(0, 73){\makebox(0, 0)[b]{$\sigma^0$}}
    \put(5, 60){\line(1, 0){30}}
    \put(34, 57){\textcircled{\tiny $2$}}
    \put(40, 70){\makebox(0, 0)[b]{$\Gamma_{P,1}$}}
    \put(45, 60){\line(1, 0){30}}
    \put(74, 57){\textcircled{\tiny $2$}}
    \put(80, 70){\makebox(0, 0)[b]{$\Gamma_{P,2}$}}
    \put(100, 55){\makebox(0, 0)[b]{$\sqcup$}}
    \put(114, 57){\textcircled{\tiny $4$}}
    \put(120, 73){\makebox(0, 0)[b]{$l^0$}}
    \end{picture}
\end{center}
\end{description}
Thus the associated 
log del Pezzo surfaces $S$ are of index $4$ by Corollary \ref{indacor}. 
Furthermore, the value $(-K_S^2)$ is equal to $8$ in any case 
by Proposition \ref{fundprop} \eqref{fundprop3}. 
\end{example}

\begin{remark}\label{ind23rmk}
Let $S$ be a log del Pezzo surface of index $a$ such that $(-K_S^2)\geq 2a$. 
\begin{enumerate}
\renewcommand{\theenumi}{\arabic{enumi}}
\renewcommand{\labelenumi}{(\theenumi)}
\item\label{ind23rmk1}
Assume that $a=2$. Then $S$ is isomorphic to the log del Pezzo surface associated 
to the fundamental triplet in the sense of \cite{N} whose type is one of the 
following: 
\begin{itemize}
\item
$[4;1,0]_0$ \,\,\, $\left((-K_S^2)=9\right)$, 
\item
$[3;1,0]_0$ \,\,\, $\left((-K_S^2)=8\right)$,
\item
$[2;1,0]_0$ \,\,\, $\left((-K_S^2)=7\right)$,
\item
$[1;1,0]_0$, $[3;1,1]_+$ \,\,\, $\left((-K_S^2)=6\right)$,
\item
$[0;1,0]_0$, $[2;1,1]_+(a,b)$ \,\,\, $\left((-K_S^2)=5\right)$,
\item
$[1;1,1]_0$, $[1;1,1]_+(a,b)$ \,\,\, $\left((-K_S^2)=4\right)$.
\end{itemize}
We note that the types $[4;1,0]_0$, $[3;1,0]_0$, $[2;1,0]_0$, $[1;1,0]_0$, $[0;1,0]_0$ 
in \cite{N} are nothing but the types 
$\langle${\rm O}$\rangle_2$, $\langle${\rm I}$\rangle_2$, $\langle${\rm II}$\rangle_2$, 
$\langle${\rm III}$\rangle_2$, $\langle${\rm IV}$\rangle_2$
in Example \ref{fEO}, respectively. 
\item\label{ind23rmk2}
Assume that $a=3$. Then $S$ is isomorphic to the log del Pezzo surface associated 
to the bottom tetrad in the sense of \cite{FY} whose type is one of the 
following: 
\begin{itemize}
\item
\textbf{$[$6;2,0$]$} \,\,\, $\left((-K_S^2)=32/3\right)$,
\item
\textbf{$[$5;2,0$]$} \,\,\, $\left((-K_S^2)=29/3\right)$,
\item
\textbf{$[$4;2,0$]$} \,\,\, $\left((-K_S^2)=26/3\right)$,
\item
\textbf{$[$3;1,0$]$} \,\,\, $\left((-K_S^2)=25/3\right)$,
\item
\textbf{$[$5;2,1$]_1$} \,\,\, $\left((-K_S^2)=8\right)$,
\item
\textbf{$[$3;2,0$]$} \,\,\, $\left((-K_S^2)=23/3\right)$, 
\item
\textbf{$[$2;1,0$]$}, \textbf{$[$4;2,1$]_{1B}$} \,\,\, $\left((-K_S^2)=22/3\right)$, 
\item
\textbf{$[$4;2,1$]_{1A}$} \,\,\, $\left((-K_S^2)=7\right)$, 
\item
\textbf{$[$4;2,2$]_{1F}$} \,\,\, $\left((-K_S^2)=20/3\right)$, 
\item
\textbf{$[$1;1,0$]$}, \textbf{$[$2;2,0$]$}, \textbf{$[$3;2,1$]_{1B}$}, 
\textbf{$[$4;2,2$]_{1E}$} \,\, $\left((-K_S^2)=19/3\right)$, 
\item
\textbf{$[$3;2,1$]_{1A}$}, \textbf{$[$4;2,2$]_{1C}$}, \textbf{$[$4;2,2$]_{1D}$}, 
\,\, $\left((-K_S^2)=6\right)$.
\end{itemize}
We note that the types \textbf{$[$6;2,0$]$}, \textbf{$[$5;2,0$]$}, \textbf{$[$4;2,0$]$}, 
\textbf{$[$3;2,0$]$}, \textbf{$[$2;2,0$]$} 
in \cite{FY} are nothing but the types 
$\langle${\rm O}$\rangle_3$, $\langle${\rm I}$\rangle_3$, $\langle${\rm II}$\rangle_3$, 
$\langle${\rm III}$\rangle_3$, $\langle${\rm IV}$\rangle_3$
in Example \ref{fEO}, respectively. Moreover, the log del Pezzo surface of 
index three associated to the bottom tetrad of type \textbf{$[$3;1,0$]$} 
is isomorphic to the weighted projective plane $\pr(1,1,3)$. 
\end{enumerate}
\end{remark}

\subsection{Special log del Pezzo surfaces}\label{ex_section2}

In this section, we determine the log del Pezzo surfaces associated to the 
$a$-fundamental multiplets of types $\langle$O$\rangle_a$, $\langle$I$\rangle_a$ 
and $\langle$II$\rangle_a$. We freely use the notation of the toric geometry in this 
section. See \cite{fulton} for example. 
We fix a lattice $N:=\Z^{\oplus 2}$ and set 
$N_\R:=N\otimes_\Z\R(\simeq\R^{\oplus 2})$.

\begin{example}\label{dEO}
Fix $a\geq 2$. It is well-known that the weighted projective plane $\pr(1,1,2a)$ is 
a log del Pezzo surface of index $a$; the associated $a$-basic pair is equal to 
$(\F_{2a}, (a-1)\sigma)$. Thus the log del Pezzo surface associated to the 
$a$-fundamental multiplet of type $\langle$O$\rangle_a$ is isomorphic to $\pr(1,1,2a)$. 
\end{example}

\begin{example}\label{dEI}
Fix $a\geq 2$. Let $\Sigma_{{\rm I}, a}$ be the complete fan in $N_\R$ such that 
the set of the generators of one-dimensional cones in $\Sigma_{{\rm I}, a}$ is 
\[
\{(1, 0), (0, 1), (-1, -2a+1), (-1, -2a)\}.
\]
Let $S_{{\rm I}, a}$ be the projective toric surface associated to the fan 
$\Sigma_{{\rm I}, a}$. Then the minimal resolution $M_{{\rm I}, a}$ of 
$S_{{\rm I}, a}$ corresponds to the complete fan $\Sigma'_{{\rm I}, a}$ 
such that the set of the generators of one-dimensional cones 
in $\Sigma'_{{\rm I}, a}$ is 
\[
\{(1, 0), (0, 1), (-1, -2a+1), (-1, -2a), (0, -1)\}.
\]
Moreover, the divisor $-aK_{M_{{\rm I}, a}/S_{{\rm I}, a}}$ is equal to 
$(a-1)V(\R_{\geq 0}(0,-1))$, where $V(\R_{\geq 0}(0,-1))$ is the torus-invariant 
prime divisor associated to the cone $\R_{\geq 0}(0,-1)\in\Sigma'_{{\rm I}, a}$. 
We note that $M_{{\rm I}, a}$ is isomorphic to the variety $\F_{2a-1}$ 
blowing upped along a point on $\sigma$. Moreover, $V(\R_{\geq 0}(0, -1))$ 
corresponds to the strict transform of $\sigma\subset\F_{2a-1}$. 
Therefore the log del Pezzo surface associated to the $a$-fundamental multiplet 
of type $\langle$I$\rangle_a$ is isomorphic to $S_{{\rm I}, a}$. 
\end{example}

\begin{example}\label{dEII1}
Fix $a\geq 2$. Let $\Sigma_{{\rm II}_1, a}$ be the complete fan in $N_\R$ such that 
the set of the generators of one-dimensional cones in $\Sigma_{{\rm II}_1, a}$ is 
\[
\{(1, 0), (0, 1), (-1, -2a+2), (-1, -2a)\}.
\]
Let $S_{{\rm II}_1, a}$ be the projective toric surface associated to the fan 
$\Sigma_{{\rm II}_1, a}$. Then the minimal resolution $M_{{\rm II}_1, a}$ of 
$S_{{\rm II}_1, a}$ corresponds to the complete fan $\Sigma'_{{\rm II}_1, a}$ 
such that the set of the generators of one-dimensional cones 
in $\Sigma'_{{\rm II}_1, a}$ is 
\[
\{(1, 0), (0, 1), (-1, -2a+2), (-1, -2a+1), (-1, -2a), (0, -1)\}.
\]
Moreover, the divisor $-aK_{M_{{\rm II}_1, a}/S_{{\rm II}_1, a}}$ is equal to 
$(a-1)V(\R_{\geq 0}(0,-1))$. 
We note that the pair $(M_{{\rm II}_1, a}, (a-1)V(\R_{\geq 0}(0,-1)))$ is isomorphic to 
the $a$-basic pair associated to the $a$-fundamental multiplet of type 
$\langle$II$\rangle_a$ such that $\#|\Delta_1|=1$. 
Therefore the log del Pezzo surface associated to the $a$-fundamental multiplet of type 
$\langle$II$\rangle_a$ such that $\#|\Delta_1|=1$ is isomorphic to $S_{{\rm II}_1, a}$. 
\end{example}

\begin{example}\label{dEII2}
Fix $a\geq 2$. Let $\Sigma_{{\rm II}_2, a}$ be the complete fan in $N_\R$ such that 
the set of the generators of one-dimensional cones in $\Sigma_{{\rm II}_2, a}$ is 
\[
\{(1, 0), (0, 1), (-1, -2a+2), (-1, -2a+1), (1, -1)\}.
\]
Let $S_{{\rm II}_2, a}$ be the projective toric surface associated to the fan 
$\Sigma_{{\rm II}_2, a}$. Then the minimal resolution $M_{{\rm II}_2, a}$ of 
$S_{{\rm II}_2, a}$ corresponds to the complete fan $\Sigma'_{{\rm II}_2, a}$ 
such that the set of the generators of one-dimensional cones 
in $\Sigma'_{{\rm II}_2, a}$ is 
\[
\{(1, 0), (0, 1), (-1, -2a+2), (-1, -2a+1), (0, -1), (1, -1)\}.
\]
Moreover, the divisor $-aK_{M_{{\rm II}_2, a}/S_{{\rm II}_2, a}}$ is equal to 
$(a-1)V(\R_{\geq 0}(0,-1))$. 
We note that the pair $(M_{{\rm II}_2, a}, (a-1)V(\R_{\geq 0}(0,-1)))$ is isomorphic to 
the $a$-basic pair associated to the $a$-fundamental multiplet of type 
$\langle$II$\rangle_a$ such that $\#|\Delta_1|=2$. 
Therefore the log del Pezzo surface associated to the $a$-fundamental multiplet of type 
$\langle$II$\rangle_a$ such that $\#|\Delta_1|=2$ is isomorphic to $S_{{\rm II}_2, a}$. 
\end{example}

From the arguments in this section, Corollary \ref{maincor} is deduced from 
Theorem \ref{mainthm} immediately.

\section{Proof of Theorem \ref{mainthm}}\label{proof_section}

In this section, we prove Theorem \ref{mainthm}. From now on, let $a\geq 4$, 
$S$ be a log del Pezzo surface of index $a$ with $(-K_S^2)\geq 2a$, 
$(M_0, E_0)$ be the associated $a$-basic pair, $(M_b, E_b; \Delta_1,\dots,\Delta_b)$ be 
an $a$-fundamental multiplet of length $b$ with $1\leq b\leq a-1$ such that 
the associated $a$-basic pair is equal to $(M_0, E_0)$, and $L_b$ be the fundamental 
divisor of the multiplet. We remark that the existence of such multiplet is proven 
in Corollary \ref{bfcor}. By Proposition \ref{hirzprop}, if $M_b\simeq\F_n$ and 
$bK_{M_b}+L_b$ is non-big and non-trivial, then we can assume that 
$\Delta_b\cap\sigma=\emptyset$. We know that 
$a(-K_S^2)=(-K_{M_b}\cdot L_b)-\sum_{i=1}^bi\deg\Delta_i$ by Proposition 
\ref{fundprop} \eqref{fundprop3}.

\subsection{Structure of $M_b$}\label{proof1_section}

In this section, we prove the following lemma. 

\begin{lemma}\label{Mblem}
$M_b$ is isomorphic to $\F_n$ for some $n\geq 0$. 
\end{lemma}

\begin{proof}
Assume not. By Proposition \ref{fundprop} \eqref{fundprop2}, we can assume that 
$M_b=\pr^2$. Set $h\in\Z_{>0}$ such that $L_b\sim hl$, where $l$ is a line. 
By Proposition \ref{fundprop}, $h\leq 3b+2$ and 
$2a^2\leq a(-K_S^2)=3h-\sum_{i=1}^bi\deg\Delta_i\leq 3(3b+2)\leq 3(3a-1)$. 
Thus $a=4$, $b=a-1=3$, $h=3b+2=11$ and $\sum_{i=1}^bi\deg\Delta_i\leq 1$. 
However, $11=(L_b\cdot E_b)=\sum_{i=1}^3i(4-i)\deg\Delta_i
\leq 4\sum_{i=1}^3i\deg\Delta_i\leq 4$. This leads to a contradiction. 
\end{proof}

Therefore, we can assume that $M_b=\F_n$. 
Set $h_0$, $h\in\Z$ such that $L_b\sim h_0\sigma+hl$. Then 
$E_b\sim(2a-h_0)\sigma+((n+2)a-h)l$ and 
$a(-K_S^2)=-nh_0+2h_0+2h-\sum_{i=1}^bi\deg\Delta_i$. 
We note that $b=\lfloor h_0/2\rfloor$ by Proposition \ref{fundprop} \eqref{fundprop2}. 
Moreover, $h\geq nh_0$ and $h_0\geq 1$ since $L_b$ is nef and big. 
In particular, $0<h_0<2a$.

\subsection{Determine the value $h_0$}\label{proof2_section}

In this section, we prove that $h_0=a+1$ and $b=\lfloor(a+1)/2\rfloor$. 
To begin with, we see the following claim.

\begin{claim}\label{h0claim}
We have $h_0\geq a+1$. 
\end{claim}

\begin{proof}
Assume that $h_0\leq a$. Since $\coeff_\sigma E_b\leq a-1$, we have 
$h\leq 2a+n(h_0-1)$. Since $h\geq nh_0$, this implies that $n\leq 2a$. 
If $h_0\leq a-1$, then 
$a(-K_S^2)\leq -nh_0+2h_0+4a+2n(h_0-1)\leq(n+2)(a-1)+4a-2n\leq 2a^2-2$, 
which leads to a contradiction. Thus $h_0=a$ (under the assumption $h_0\leq a$). 
If $n\leq 2a-5$, then $a(-K_S^2)\leq -na+2a+4a+2n(a-1)=(a-2)n+6a\leq 2a^2-3a+10$, 
which leads to a contradiction. Thus $2a-4\leq n\leq 2a$. In particular, $n\geq a$. 
Moreover, $h\geq 2a+(a-2)n$ since $h\geq an\geq an+2(a-n)$. 
If $h=2a+(a-2)n$, then $a=n=4$ and $h=16$. However, in this case, 
$a(-K_S^2)\leq -4\cdot 4+2\cdot 4+2\cdot 16<32=2\cdot 4^2$, which leads to 
a contradiction. Thus we have $h>2a+(a-2)n$. Since 
$E_b\sim a\sigma+((n+2)a-h)l$, $\coeff_\sigma E_b\leq a-1$ and $(n+2)a-h<2n$, 
there exists a section $C\leq E_b$ apart from $\sigma$. 
By Proposition \ref{fundprop} \eqref{fundprop3}, 
\begin{eqnarray*}
h & \leq & (L_b\cdot C)=\sum_{i=1}^bi\deg(\Delta_i\cap C^i)
\leq\sum_{i=1}^bi\deg\Delta_i\\
 & = & -nh_0+2h_0+2h-a(-K_S^2)\leq a(2-n)+2h-2a^2. 
\end{eqnarray*}
Thus $h\geq 2a^2+a(n-2)$. However, we know that $h\leq (n+2)a-n$, which leads to 
a contradiction. 
\end{proof}

By Claim \ref{h0claim}, $h_0\geq a+1$. Since $E_b$ is effective, we have $h\leq (n+2)a$. 
Since $h\geq nh_0$, we have $n\leq 2a/(h_0-a)$. 
From now on, we assume that $h_0\geq a+2$. Then 
\begin{eqnarray*}
a(-K_S^2) & \leq & n(2a-h_0)+2h_0+4a-\sum_{i=1}^bi\deg\Delta_i\\
 & \leq & \frac{2h_0^2}{h_0-a}-\sum_{i=1}^bi\deg\Delta_i\leq\frac{2h_0^2}{h_0-a}
\leq (a+2)^2.
\end{eqnarray*}
Thus $a=4$. In this case, $h_0=6$ or $7$. 
Assume that $n\leq 2$. Then 
$a(-K_S^2)\leq n(2a-h_0)+2h_0+4a-\sum_{i=1}^bi\deg\Delta_i
\leq 2(8-h_0)+2h_0+16=32$. Thus $n=2$, $h=(n+2)a=16$ and 
$\sum_{i=1}^bi\deg\Delta_i=0$. However, in this case, $(L_b\cdot E_b)=2(8-h_0)^2\neq 0$, 
which leads to a contradiction. Hence $n\geq 3$. 
Thus $h_0=6$, $b=3$, and $n=3$ or $4$ since $n\leq 2a/(h_0-a)$. 
Assume that $h<a(n+2)=4(n+2)$. Since $a(-K_S^2)=6(2-n)+2h-\sum_{i=1}^bi\deg\Delta_i
\geq 32$, we have $h\geq 4(n+2)-2$ and $\sum_{i=1}^bi\deg\Delta_i\leq 2$. 
We note that there exists a fiber $l\leq E_b$. By Proposition 
\ref{fundprop} \eqref{fundprop3}, 
$6=h_0=(L_b\cdot l)=\sum_{i=1}^bi\deg(\Delta_i\cap l^i)\leq\sum_{i=1}^bi\deg\Delta_i\leq 2$, which leads to contradiction. Thus $h$ must be equal to $4(n+2)$. In particular, 
$L_3\sim 6\sigma+4(n+2)l$, $E_3=2\sigma$, $\sum_{i=1}^3i(4-i)\deg\Delta_i=16-4n$ 
and $\sum_{i=1}^3i\deg(\Delta_i\cap\sigma^i)=8-2n$. 

Assume that $n=3$. Then $\Delta_1=\emptyset$, $\Delta_3=\emptyset$, 
$\deg\Delta_2=1$ and $\Delta_2\subset\sigma^2$. Then the associated log del Pezzo 
surface is isomorphic to $S_{{\rm I}, 2}$ having described in Example \ref{dEI}. 
However, the index of $S_{{\rm I}, 2}$ is equal to two, a contradiction. 

Assume that $n=4$. Then $\Delta_1=\emptyset$, $\Delta_2=\emptyset$ and 
$\Delta_3=\emptyset$. then the associated log del Pezzo surface is isomorphic to 
$\pr(1,1,4)$. However, the index of $\pr(1,1,4)$ is equal to two (see Example \ref{dEO}), 
a contradiction. 

Consequently, we have $h_0=a+1$ and $b=\lfloor h_0/2\rfloor=\lfloor(a+1)/2\rfloor$.

\subsection{Determine the values $h$ and $n$}\label{proof3_section}

In this section, we prove that $h\geq 2a^2-2a-2$ and $2a-4\leq n\leq 2a$. 
To begin with, we see the following claim.

\begin{claim}\label{sigmaclaim}
We have $\sigma\leq E_b$. 
\end{claim}

\begin{proof}
If $\sigma\not\leq E_b$, then $(n+2)a-h\geq (a-1)n$ since 
$E_b\sim(a-1)\sigma+((n+2)a-h)l$. However, in this case, 
$a(-K_S^2)\leq(-n+2)(a+1)+2(2a+n)=n(1-a)+6a+2\leq 6a+2<2a^2$, which leads to 
a contradiction. 
\end{proof}

By Proposition \ref{fundprop} \eqref{fundprop3}, we have 
\begin{eqnarray*}
-n(a+1)+h & = & (L_b\cdot\sigma)=\sum_{i=1}^bi\deg(\Delta_i\cap\sigma^i)
\leq\sum_{i=1}^bi\deg\Delta_i\\
 & = & (-n+2)(a+1)+2h-a(-K_S^2)\\
 & \leq & (-n+2)(a+1)+2h-2a^2.
\end{eqnarray*}
Hence $h\geq 2a^2-2a-2$. Since $(a+1)n\leq h\leq(n+2)a$, we have 
$2a-4\leq n\leq 2a$.

\subsection{The case $h=(n+2)a$}\label{proof4_section}

We consider the case $h=(n+2)a$. In this case, 
$L_b\sim(a+1)\sigma+(n+2)al$, $E_b=(a-1)\sigma$, 
$2a-n=(L_b\cdot\sigma)=\sum_{i=1}^bi\deg(\Delta_i\cap\sigma^i)$, 
$(a-1)(2a-n)=(L_b\cdot E_b)=\sum_{i=1}^bi(a-i)\deg\Delta_i$ 
and $a(-K_S^2)=(n-(2a-4))(a-1)+2a^2+6-\sum_{i=1}^bi\deg\Delta_i$. 
In particular, $\sum_{i=1}^bi\deg\Delta_i\leq 6+(n-(2a-4))(a-1)$. 

If $\Delta_2=\emptyset,\dots,\Delta_b=\emptyset$, then $\deg\Delta_1=2a-n$ and 
$\Delta_1\subset\sigma^1$. For the case $n=2a$ (resp.\ $n=2a-1$, $n=2a-2$, 
$n=2a-3$, $n=2a-4$), the $a$-fundamental multiplet 
$(M_b, E_b; \Delta_1,\dots,\Delta_b)$ is of type 
$\langle$O$\rangle_a$ (resp.\ $\langle$I$\rangle_a$, $\langle$II$\rangle_a$, 
$\langle$III$\rangle_a$, $\langle$IV$\rangle_a$). 

From now on, we assume that $\Delta_{i_0}\neq\emptyset$ for some $2\leq i_0\leq b$. 
We can assume that $\Delta_j=\emptyset$ for any $j\geq i_0+1$. 
Then $\Delta_{i_0}\cap\sigma^{i_0}\neq\emptyset$. Thus 
$i_0\leq 2a-n\leq 4$. 

Assume that $i_0=4$. Then $2a-n=4$, $\deg(\Delta_4\cap\sigma^4)=1$ and 
$\sum_{i=1}^bi\deg\Delta_i\leq 6$. Let $\Gamma\subset M_3$ be the 
$\pi_4$-exceptional $(-1)$-curve. Then $\Gamma\leq E_3$ and $\sum_{i=1}^3i\deg(\Delta_i\cap\Gamma^i)=4$ by Lemma \ref{FS1}. This leads to 
a contradiction. 

Assume that $i_0=3$. Since $b\geq 3$, we have $a\geq 5$. 
Moreover, $\deg(\Delta_3\cap\sigma^3)=1$, $\deg(\Delta_2\cap\sigma^2)=0$ 
and $\deg(\Delta_1\cap\sigma^1)=2a-n-3\leq 1$. If $\deg\Delta_3\geq 2$ (and $i_0=3$), 
then $a=5$ and $\deg\Delta_3=2$ by Lemma \ref{FS1} \eqref{FS14}. 
However, this contradicts to the assumption $\Delta_b\cap\sigma=\emptyset$. 
Assume that $\deg\Delta_3=1$ (and $i_0=3$). 
Let $\Gamma\subset M_2$ be the $\pi_3$-exceptional $(-1)$-curve. 
Since $\coeff_\Gamma E_2=2$, we have $\Delta_2=\emptyset$ and $|\Delta_1|=\{P\}$, 
where $P=\sigma^1\cap\Gamma^1$. 
Since $\sum_{i=1}^2i\deg(\Delta_i\cap\Gamma^i)=3$, we have $\mult_P(\Delta_1\cap\Gamma^1)=3$. This contradicts to Lemma \ref{FS2} and 
the fact $a\geq 5$. 

Assume that $i_0=2$. If $\deg(\Delta_2\cap\sigma^2)=1$, then $\deg\Delta_2=1$ 
by Lemma \ref{FS1} \eqref{FS14}. In this case, $\coeff_\Gamma E_1=1$ and 
$\deg(\Delta_1\cap\Gamma)=2$, where $\Gamma\subset M_1$ is the 
$\pi_2$-exceptional $(-1)$-curve. For $P:=\sigma^1\cap\Gamma$, 
$\mult_P(\Delta_1\cap\Gamma)=2$ and $\mult_P(\Delta_1\cap\sigma^1)=1$. 
This contradicts to Lemma \ref{FS2}. Thus $\deg(\Delta_2\cap\sigma^2)=2$. 
In particular, $n=2a-4$ and $\Delta_1\cap\sigma^1=\emptyset$. 
Thus $a=4$, $|\Delta_2|=\{P\}$ such that $\mult_P\Delta_2=3$, 
$\mult_P(\Delta_2\cap\sigma^2)=2$ and $\Delta_1=\emptyset$ by Lemma \ref{FS3} 
and the fact $\sum_{i=1}^bi\deg\Delta_i\leq 6$. Then the $4$-fundamental 
multiplet $(M_2, E_2; \Delta_1, \Delta_2)$ is of type $\langle$B$\rangle_4$.

\subsection{The case $h=(n+2)a-1$}\label{proof5_section}

We consider the case $h=(n+2)a-1$. In this case, $n\leq 2a-1$, 
$L_b\sim(a+1)\sigma+((n+2)a-1)l$, $E_b=(a-1)\sigma+l$, 
$2a-n-1=(L_b\cdot\sigma)=\sum_{i=1}^bi\deg(\Delta_i\cap\sigma^i)$, 
$a+1=(L_b\cdot l)=\sum_{i=1}^bi\deg(\Delta_i\cap l^i)$, 
$4a-2-(a-1)(n-(2a-4))=(L_b\cdot E_b)=\sum_{i=1}^bi(a-i)\deg\Delta_i$ and 
$a(-K_S^2)=(n-(2a-4))(a-1)+2a^2+4-\sum_{i=1}^bi\deg\Delta_i$. 
In particular, $\sum_{i=1}^bi\deg\Delta_i\leq 4+(n-(2a-4))(a-1)$. Hence $n\geq 2a-3$. 

Assume that $\Delta_i\cap\sigma^i=\emptyset$ for all $i\geq 2$. Then 
$\Delta_i=\emptyset$ for all $i\geq 2$. However, by Lemma \ref{FS2}, 
$P\not\in\Delta_1$, where $P=\sigma^1\cap l^1$. Thus $\Delta_1\cap l^1=\emptyset$, 
which leads to a contradiction. Thus $n=2a-3$, $\deg(\Delta_2\cap\sigma^2)=1$ 
and $\deg(\Delta_i\cap\sigma^i)=0$ for all $i\neq 2$. Then $|\Delta_2|=\{P\}$ with 
$P=\sigma^2\cap l^2$ by Lemma \ref{FS2}. Moreover, since $E_1$ is effective, 
by Example \ref{E2}, $\deg\Delta_2=\deg(\Delta_2\cap l^2)$ and 
$\deg\Delta_2\leq 3$ (if $a=4$); $\deg\Delta_2\leq 2$ (if $a=5$); $\deg\Delta_2=1$ 
(if $a\geq 6$), unless $(a, \deg\Delta_2, \deg(\Delta_2\cap l^2))=(4,2,1)$. 
If $(a, \deg\Delta_2, \deg(\Delta_2\cap l^2))=(4,2,1)$, then 
$\Delta_1\cap l^1=\emptyset$, a contradiction. 
Thus $\deg(\Delta_2\cap l^2)=\deg\Delta_2$. If $\deg\Delta_2=1$, then 
$\deg\Delta_1=(a+3)/(a-1)$ and $\deg(\Delta_1\cap l^1)=a-1$, a contradiction. 
If $\deg\Delta_2=3$, then $\deg(\Delta_1\cap l^1)=a-5=-1$, a contradiction. 
Thus $\deg\Delta_2=2$. In this case, $\deg\Delta_1=(-a+7)/(a-1)$ and 
$\deg(\Delta_1\cap l^1)=a-3$. Thus $a=4$, $\deg\Delta_1=\deg(\Delta_1\cap l^1)=1$. 
However, in this case, $E_1=3\sigma^1+2\Gamma_{P, 1}+\Gamma_{P, 2}+l^1$. Thus 
$\Delta_1\cap l^1=\emptyset$. This leads to a contradiction. Therefore 
$h\neq(n+2)a-1$.

\subsection{The case $h\leq(n+2)a-2$}\label{proof6_section}

We consider the case $h\leq(n+2)a-2$. In this case, $n\leq 2a-2$. 
Assume that $h\leq(n+2)a-3$. Then $n=2a-3$ and $h=2a^2-a-3$ since 
$h\geq 2a^2-2a-2$. In this case, $E_b\sim(a-1)\sigma+3l$ and 
$a(-K_S^2)=2a^2+a-1-\sum_{i=1}^bi\deg\Delta_i$. Hence there exists a fiber $l\leq E_b$. 
Thus $a+1=(L_b\cdot l)=\sum_{i=1}^bi\deg(\Delta_i\cap l^i)\leq
\sum_{i=1}^bi\deg\Delta_i\leq a-1$, which leads to a contradiction. Thus $h$ must 
be equal to $(n+2)a-2$. In this case, 
$L_b\sim(a+1)\sigma+((n+2)a-2)l$, $E_b\sim(a-1)\sigma+2l$, 
$2a-n-2=(L_b\cdot\sigma)=\sum_{i=1}^bi\deg(\Delta_i\cap\sigma^i)$, 
$4a-(a-1)(n-(2a-4))=(L_b\cdot E_b)=\sum_{i=1}^bi(a-i)\deg\Delta_i$ and 
$a(-K_S^2)=(n-(2a-4))(a-1)+2a^2+2-\sum_{i=1}^bi\deg\Delta_i$. 
In particular, $\sum_{i=1}^bi\deg\Delta_i\leq 2+(n-(2a-4))(a-1)$. 
We note that there exists a fiber $l\leq E_b$ and 
$a+1=(L_b\cdot l)=\sum_{i=1}^bi\deg(\Delta_i\cap l^i)\leq\sum_{i=1}^bi\deg\Delta_i
\leq 2a$ for such $l$. Thus $E_b=(a-1)\sigma+2l$. 
Moreover, we have $n\geq 2a-3$. Indeed, if $n=2a-4$, then 
$a+1\leq\sum_{i=1}^bi\deg\Delta_i\leq 2$, a contradiction. 

If $n=2a-2$, then $0=(L_b\cdot\sigma)=\sum_{i=1}^bi\deg(\Delta_i\cap\sigma^i)$. 
Since $a-b\geq 2$, we have $a-b=2$, $\Delta_b\subset l^b\setminus\sigma^b$, 
$\Delta_1=\emptyset,\dots,\Delta_{b-1}=\emptyset$ by Lemma \ref{FS1} \eqref{FS13}. 
In particular, $b+3=a+1=b\deg\Delta_b$. Thus $a=5$, $b=3$ and $\deg\Delta_3=2$. 
Then the $5$-fundamental multiplet $(M_3, E_3; \Delta_1, \Delta_2, \Delta_3)$ is 
of type $\langle$A$\rangle_5$. 

The remaining case is that $n=2a-3$. Since $a+1=\sum_{i=1}^bi\deg(\Delta_i\cap l^i)
\leq\sum_{i=1}^bi\deg\Delta_i\leq a+1$, we have $\sum_{i=1}^bi\deg\Delta_i=a+1$ and 
$\Delta_i\subset l^i$ for any $i$. 
Since $1=\sum_{i=1}^bi\deg(\Delta_i\cap\sigma^i)$, $\deg(\Delta_1\cap\sigma^1)=1$ and 
$\Delta_i\cap\sigma^i=\emptyset$ for any $i\geq 2$. 
Set $P=\sigma^1\cap l^1$. 
Since $\sum_{i=1}^bi\deg(\Delta_i\cap l^i)=\sum_{i=1}^bi\deg\Delta_i$, 
we have $P\in\Delta_1$. 
By Lemma \ref{FS2}, either $(a, \mult_P\Delta_1)=(4,3)$ or $(5,2)$ holds. 
If $(a, \mult_P\Delta_1)=(5,2)$, then $\deg\Delta_2=2$ and $\Delta_3=\emptyset$ 
since $6=\sum_{i=1}^3i\deg\Delta_i$. However, $\Delta_2=\emptyset$ by 
Lemma \ref{FS1} \eqref{FS13}, a contradiction. 
Thus $(a, \mult_P\Delta_1)=(4,3)$. In this case, $\deg\Delta_2=1$ since 
$\sum_{i=1}^2i\deg\Delta_i=5$. Then the $4$-fundamental multiplet 
$(M_2, E_2; \Delta_1, \Delta_2)$ is of type $\langle$C$\rangle_4$. 

As a consequence, we have complete the proof of Theorem \ref{mainthm}.

\end{document}